\documentclass{gtpart}     % Basic GT/GTM/AGT style
\usepackage{epsfig,color} 
\usepackage{mathtools,amsmath}

\title{A positive proportion of elements of mapping class groups\\ is pseudo-Anosov}

\author{Mar\'ia Cumplido}
\givenname{Maria}
\surname{Cumplido}
\address{Mar\'ia Cumplido, UFR Math\'ematiques, Universit\'e de Rennes 1, France, and Departamento de \'Algebra, Universidad de Sevilla, Spain}
\email{maria.cumplidocabello@univ-rennes1.fr}

\author{Bert Wiest}
\givenname{Bert}
\surname{Wiest}
\address{Bert Wiest, UFR Math\'ematiques, Universit\'e de Rennes 1,  France}
\email{bertold.wiest@univ-rennes1.fr}
\volumenumber{}
\issuenumber{}
\publicationyear{}
\papernumber{}
\startpage{}
\endpage{}
\doi{}
\MR{}
\Zbl{}
\received{}
\revised{}
\accepted{}
\published{}
\publishedonline{}
\proposed{}
\seconded{}
\corresponding{}
\editor{}
\version{}

\makeautorefname{notation}{Notation}%

\newtheorem{theorem}{Theorem}%[section]
\newtheorem{lemma}[theorem]{Lemma} 

\newtheorem{corollary}[theorem]{Corollary}
\theoremstyle{definition}
    
\newtheorem{remark}[theorem]{Remark}

\newtheorem{example}[theorem]{Example}

\def\Z{{\mathbb Z}}
\def\R{{\mathbb R}}
\def\co{\colon \thinspace}

\def\CC{\mathcal{CC}}

\renewcommand{\phi}{\varphi}
\def\F{\mathcal F}
\def\G{\mathcal G}
\def\H{\mathcal H}

\begin{document}

\begin{abstract}
In the Cayley graph of the mapping class group of a closed surface, with respect to any generating set, we look at a ball of large radius centered on the identity vertex, and at the proportion among the vertices in this ball representing pseudo-Anosov elements. A well-known conjecture states that this proportion should tend to one as the radius tends to infinity. We prove that it stays bounded away from zero. We also prove similar results for a large class of subgroups of the mapping class group.
\end{abstract}

\maketitle

%======================================

\section{Pseudo-Anosovs in mapping class groups}

%The aim of this paper is to prove that picking an element at random in a large ball in the Cayley graph of a mapping class group, with any generating set, yields a pseudo-Anosov element with a probability which stays bounded away from zero as the radius of the ball tends to infinity. 
Let $\G$ be the mapping class group of a closed surface~$\Sigma$ (with genus$(\Sigma)\geqslant 2$), equipped with any finite generating set~$S$.
It is a widely  held belief that ``most'' elements of~$\G$ are pseudo-Anosov. In order to make this statement precise, we have to specify a way picking a ``random element'' of~$\G$. 

There are at least two standard methods for doing so, both involving the Cayley graph~$\Gamma$ of~$\G$ with respect to~$S$.  
A first method is to perform a random walk in the Cayley graph~$\Gamma$  of~$\G$, starting at the identity-vertex. In this framework, the belief is that the probability of obtaining a pseudo-Anosov element tends (exponentially quickly) to~$1$ as the length of the random walk tends to infinity. This belief has been proven to be correct, and it has been generalised far beyond the realm of mapping class groups, in \cite{Rivin, Maher, Sisto}.

For the second method, consider the ball $B_\Gamma(1,R)$ in~$\Gamma$ of radius $R$ and centered on the identity vertex $1$. The belief now is that the proportion of pseudo-Anosov elements among the vertices in this ball tends to~$1$ as the radius tends to infinity:
$$
\lim_{R\to\infty} \frac{|\text{pseudo-Anosov elements of }\G \ \cap \  B_\Gamma(1,R)|}{|B_\Gamma(1,R)|} \ = \ 1
$$
This statement remains a conjecture despite some progress in~\cite{CarusoWiest}.
Our main result is

\begin{corollary}\label{C:main}
$$
\liminf_{R\to\infty} \ \frac{|\text{pseudo-Anosov elements of }\G \ \cap \  B_\Gamma(1,R)|}{|B_\Gamma(1,R)|} \ > \ 0
$$
\end{corollary}

Our proof is actually very easy, and it uses only well-known ingredients -- notably, a classical theorem of Albert Fathi~\cite{Fathi}, and the acylindrical hyperbolicity of the mapping class group action on the curve complex~\cite{BowditchTightGeodCC}. 
Therefore it seems likely that Corollary~\ref{C:main}, and our proof, is well-known to some experts. However, we have not been able to locate it in the literature, and believe it deserves to be written down.

We will deduce Corollary~\ref{C:main} from the following, which is our main technical result:
%Here is our main technical result, which immediately implies Corollary~\ref{C:main}:

\begin{theorem}[Positive density of pseudo-Anosovs]\label{T:main} 
There exists a finite subset~$\F$ of~$\G$ such that for any element $g$ of~$\G$, at least one of the mapping classes in $\{f\circ g \ | \ f\in\F\}$ is pseudo-Anosov.
\end{theorem}

%==============================================

\section{Proofs}\label{S:Proofs}

{\bf Proof of Theorem~\ref{T:main} } \ In what follows, by a \emph{curve} we always mean an isotopy class of simple closed curves. We start by recalling the result from~\cite{Fathi}~: 

\begin{theorem}[Fathi]\label{T:Fathi}
If $g\in\G$ and if $c$ is a simple closed curve on~$\Sigma$ such that the curves $\{g^n(c) \ | \ n\in\Z\}$ together fill~$\Sigma$, then $T_c^k\circ g$ is always pseudo-Anosov, except for at most seven consecutive values of~$k$ (where $T_c$ denotes the Dehn twist along the curve~$c$). In particular, either $g$ or $T_c^7\circ g$ (or both) are pseudo-Anosov. 
\end{theorem}

\begin{lemma}\label{L:ab}
There are two curves $a,b$ in $\Sigma$ such that for all but finitely many %, non-pseudo-Anosov 
elements $g$ of~$\G$, either $a$ and~$g(a)$ together fill~$\Sigma$, or $b$ and $g(b)$ together fill~$\Sigma$ (or both). 
\end{lemma}

{\bf Proof of Lemma~\ref{L:ab} } %Lemma~\ref{L:ManyFilling} }
Recall that two curves together fill $\Sigma$ if and only if, in the curve complex $\CC(\Sigma)$, they are at distance $\geqslant 3$.

Next we recall a fundamental result due to Bowditch \cite{BowditchTightGeodCC}: the mapping class group~$\G$ acts \emph{acylindrically} on the curve complex, i.e., for any $r\geqslant 0$, there exist $R(r), \ N(r)\geqslant 0$ so that for any two vertices $a$, $b$ of the curve complex with $d_{\CC}(a,b)>R(r)$ there are at most $N(r)$ distinct elements $g$ of $\G$ such that $d_{\CC}(a,g(a))<r$ and $d_{\CC}(b,g(b))<r$. 

We shall apply this result in the case~$r=3$.
We simply choose arbitrarily two curves $a$ and $b$ in~$\Sigma$, the only restriction being that $d_{\CC}(a,b)>R(3)$. Bowditch's result says that for all but finitely many (at most~$N(3)$) elements~$g$ of~$\G$, the action of~$g$ on the curve complex displaces least one of the two vertices, $a$ or~$b$, by at least~$3$. 
This completes the proof of Lemma~\ref{L:ab}. \hfill$\Box$ 

Now we define $\F'=\{1_\G\cup T_a^7\cup T_b^7\}$. Fathi's theorem and Lemma~\ref{L:ab} together imply the following: 
for every element $g$ of~$\G$, apart from the finitely many exceptional ones, at least one of the mapping classes in $\{f\circ g \ | \ f\in\F'\}$ is pseudo-Anosov.

Now we have to deal with the finitely many exceptional elements of~$\G$, say $\{g_1,\ldots,g_N\}$, where $N\leqslant N(3)$. For each $g_i$, we choose one $f_i\in\G$ such that $f_i\circ g_i$ is not an exceptional element. Then for every~$i$, at least one of the mapping classes in $\{f\circ f_i\circ g_i \ | \ f\in \F'\}$ is pseudo-Anosov. Thus we have completed the proof of Theorem~\ref{T:main}, with
$$
\F = \left\{ 1, T_a^7, T_b^7 \right\} \bigcup_{i=1,\ldots,N} \left\{ f_i, \ T_a^7\circ f_i, \ T_b^7\circ f_i \right\}
$$

%In order to do this, we take $\C$ to consist of the two elements $a$ and~$b$ of~$\C'$, plus finitely many other curves, one for each exceptional element of~$\G$. 
%Specifically, if $g\in\G$ is such an element ($d_{\CC}(a,g(a))<3$ and $d_{\CC}(b,g(b))<3$), then we simply add to~$\C$ any one curve $c$ which satisfies $d_{\CC}(c,g(c))\geqslant 3$. Such a curve exists, because the only element of $\G$ which acts on the curve complex in such a way that all points are moved by a bounded amount (for instance by less than 3) is the identity element. (In other words, the natural map from the mapping class group of $\Sigma$ to the quasi-isometry group of $\CC(\Sigma)$ is injective.) For a proof, see the proof of Corollary 1.1 in~\cite{RafiSchleimerRigid}). This completes the proof of the Claim, and of  Theorem~\ref{T:main}. \hfill$\Box$

{\bf Proof of Corollary~\ref{C:main} } \ Let $R'=\max \ \{ d_\Gamma(1,f) \ | \ f\in\F\}+1$. Theorem~\ref{T:main} tells us that the union of all the balls of radius~$R'$ centered on the pseudo-Anosov vertices of~$\Gamma$ is the whole Cayley graph~$\Gamma$. 

Now for $R>R'$, the centers of the balls of radius~$R'$ covering $B_\Gamma(1,R-R')$ have to lie in $B_\Gamma(1,R)$. 
Thus the union of the balls of radius~$R'$ around the pseudo-Anosov vertices of~$B_\Gamma(1,R)$ contains $B_\Gamma(1,R-R')$. We obtain
$$
|B_\Gamma(1,R')| \ \cdot \ |\text{pseudo-Anosov elements of }\G \ \cap \  B_\Gamma(1,R)| \ \geqslant \ |B_\Gamma(1,R-R')| 
$$
and hence
$$
\frac{|\text{pseudo-Anosov elements of }\G \ \cap \  B_\Gamma(1,R)|}{|B_\Gamma(1,R)|} \geqslant \frac{|B_\Gamma(1,R-R')|}{|B_\Gamma(1,R)|}\cdot \frac{1}{|B_\Gamma(1,R')|}
$$
Both factors can be bounded below independently of~$R$. Indeed, since all vertices of~$\Gamma$ are of valence at most~$2|S|$, we have the very rough estimate
$$
\geqslant \frac{1}{2|S|+(2|S|)^2+\ldots+(2|S|)^{R'}} \cdot \frac{1}{1+2|S|+(2|S|)^2+\ldots+(2|S|)^{R'}} \geqslant \frac{1}{(2|S|)^{2(R'+1)}}
$$
Thus the $\liminf$ studied in Corollary~\ref{C:main} is at least $(2\, |S|)^{-2(R'+1)}$. \hfill$\Box$

\begin{remark}\label{R:AltProof}
After a first version of this paper appeared as a preprint, Mladen Bestvina kindly pointed out to us that there is a completely different proof of Theorem~\ref{T:main}. This proof has the virtue of also applying to other contexts, notably to proving that fully irreducible elements have positive density in~$Out(F_n)$. 
The key is the existence~\cite{BestvinaFujiwara} of homogeneous quasi-morphisms $\phi\co \G\to\R$ which are unbounded, but so that $|\phi|$ is uniformly bounded by some constant~$C$ on all reducible and periodic elements. 
Now simply fix some quasi-morphism~$\phi$, and some element $f$ of~$\G$ with $\phi(f)>C+\Delta$ and $\phi(f^{-1})<-C-\Delta$ (where $\Delta$ denotes the defect of~$\phi$). Then for any $g\in\G$,  either $f^{-1}\circ g$ or $f\circ g$ is pseudo-Anosov. For details in the $Out(F_n)$-context see~\cite{BestvinaFeighn}, particularly Remark~4.33.
\end{remark}
\bigskip

%==============================================

\section{Pseudo-Anosovs in subgroups of mapping class groups}

Throughout this section, we take $\H$ to be a subgroup of $\G$, and we take $\Gamma_{\hspace{-2pt}\H}$ to be the Cayley graph of~$\H$ with respect to any fixed finite generating set. We shall be interested in the case when $\H$ satisfies the following condition: 

{\bf Condition $(*)$ } There are vertices $a,b$ of the curve complex of~$\mathcal{CC}(\Sigma)$ and integers $k_a,k_b\in\Z$ such that 
\begin{itemize}
\item $d_{\mathcal{CC}}(a,b)\geqslant R(3)$ and
\item $T_a^{k_a}, T_b^{k_b}\in \H$
\end{itemize}
were $R(3)$ is the number used in the proof of Lemma~\ref{L:ab} in Section~\ref{S:Proofs}. 

\begin{example} Condition~$(*)$ is for instance satisfied
\begin{itemize}
\item if $\H$ is a finite index subgroup of~$\G$, or 
\item if $\H$ is the Torelli group -- see \cite[Chapter 6]{FarbMargalitPrimer}. Note that the Torelli group contains in particular all Dehn twists along separating curves of~$\Sigma$, or 
\item if $\H$ is any finite index subgroup of the Torelli group.
\end{itemize}
\end{example}

\begin{theorem}\label{T:MainSubgroups}
Suppose $\H$ is a subgroup of $\G$, the mapping class group of~$\Sigma$, satisfying Condition~$(*)$ above. Then pseudo-Anosovs have positive density in~$\H$: there exists a finite subset $\F$ of~$\H$ such that for any element $g$ of~$\H$, at least one of the mapping classes in $\{f\circ g \ | \ f\in\F\}$ is pseudo-Anosov. Moreover, 
$$
\liminf_{R\to\infty} \ \frac{|\text{pseudo-Anosov elements of }\H \ \cap \  B_{\Gamma_{\hspace{-1pt}\H}}(1,R)|}{|B_{\Gamma_{\hspace{-1pt}\H}}(1,R)|} \ > \ 0
$$
\end{theorem}

\begin{proof} The proof follows very closely the proof of Theorem~\ref{T:main}. This time we choose the vertices $a$ and $b$ of~$\Sigma$, as well as integers $k_a$ and $k_b$ according to Condition~$(*)$. Possibly after replacing $k_a$ and $k_b$ by integer multiples of themselves, we can assume moreover that $k_a,k_b\geqslant 7$.

Defining $\F'=\{1_\H\cup T_a^{k_a}\cup T_b^{k_b}\}$,  Fathi's theorem and Lemma~\ref{L:ab} now imply: 
for every element $g$ of~$\H$, apart from the finitely many exceptional ones, at least one of the mapping classes in $\{f\circ g \ | \ f\in\F'\}$ is pseudo-Anosov.

For each exceptional element $g_i\in\H$, we choose one $f_i\in\H$ such that $f_i\circ g_i\in\H$ is non-exceptional. Now with the exact same argument as in the proof of Theorem~\ref{T:main} we have: for every $g\in\H$, at least one of the mapping classes in $\{f\circ g \ | \ f\in \F'\}$ is pseudo-Anosov, with
$$
\F = \left\{ 1, T_a^{k_a}, T_b^{k_b} \right\} \bigcup_{i=1,\ldots,N} \left\{ f_i, \ T_a^{k_a}\circ f_i, \ T_b^{k_b}\circ f_i \right\}
$$
Finally, the deduction that the limit inferior of the proportion of pseudo-Anosov elements in large balls in the Cayley graph of~$\H$ is strictly positive works literally as the proof of Corollary~\ref{C:main}, just replacing $\G$ by~$\H$ and $\Gamma$ by~$\Gamma_{\hspace{-1pt}\H}$ throughout.
\end{proof}

\begin{remark} The proof of Theorem~\ref{T:main} did not use the full strength of Bowditch's theorem~\cite{BowditchTightGeodCC} -- knowing that there exists one element of~$\G$ which acts in a weakly propertly discontinuous (WPD) manner on the curve complex \cite{PrzytyckiSisto} was more than sufficient. However, for our proof of Theorem~\ref{T:MainSubgroups}, Bowditch's theorem is needed almost in its full strength.
\end{remark}

{\bf Acknowledgements } We thank Anna Lenzhen for a helpful conversation,  and Mladen Bestvina for the email mentioned in Remark~\ref{R:AltProof}. Mar\'ia Cumplido  was supported by a PhD contract founded by Universit\'{e} Rennes 1, Spanish Projects MTM2013-44233-P, MTM2016-76453-C2-1-P and FEDER and the French-Spanish mobility programme "M\'{e}rim\'{e}e 2015".  

%==============================================

\end{document}